\documentclass[10pt]{amsart}
\pdfoutput=1

\usepackage{microtype}
\usepackage[colorlinks,citecolor=black,linkcolor=black,urlcolor=black,pdftitle={The stealthiest particle configurations in 8 and 24 dimensions},pdfauthor={Alex Cohen, Henry Cohn, and Maryna Viazovska}]{hyperref}
\newcommand{\arXiv}[1]{arXiv:\href{https://arxiv.org/abs/#1}{#1}}
\usepackage{doi}

\newtheorem{theorem}{Theorem}[section]
\newtheorem{proposition}[theorem]{Proposition}
\newtheorem{lemma}[theorem]{Lemma}

\theoremstyle{definition}
\newtheorem{definition}[theorem]{Definition}

\numberwithin{equation}{section}

\newcommand{\stealth}{\operatorname{stealth}}
\newcommand{\R}{\mathbb{R}}
\newcommand{\Z}{\mathbb{Z}}
\newcommand{\E}{\mathbb{E}}
\newcommand{\cS}{\mathcal{S}}
\newcommand{\cP}{\mathcal{P}}
\newcommand{\cQ}{\mathcal{Q}}
\newcommand{\cC}{\mathcal{C}}
\newcommand{\cV}{\mathcal{V}}
\newcommand{\supp}{\operatorname{supp}}
\newcommand{\Cov}{\operatorname{Cov}}
\newcommand{\Var}{\operatorname{Var}}
\newcommand{\vol}{\operatorname{vol}}
\newcommand{\density}{\operatorname{density}}

\title[The stealthiest particle configurations in $8$ and $24$ dimensions]{The stealthiest particle configurations\\in $8$ and $24$ dimensions}

\author{Alex Cohen}
\address{Courant Institute, New York University, New York, NY, USA}
\email{alexcohen@nyu.edu}

\author{Henry Cohn}
\address{Department of Mathematics, Massachusetts Institute of Technology, Cambridge, MA, USA}
\email{cohn@mit.edu}

\author{Maryna Viazovska}
\address{Institute of Mathematics, \'Ecole Polytechnique F\'ed\'erale de Lausanne\\
Lausanne, Switzerland}
\email{viazovska@gmail.com}

\begin{document}

\begin{abstract}
We show that the $E_8$ and Leech lattices generate the stealthiest isometry-invariant, locally square-integrable point processes of intensity~$1$ in $\R^8$ and $\R^{24}$, respectively, and that they are the unique stealthiest processes. We also apply the proof technique to sphere packing and energy minimization, and we construct a $20$-dimensional periodic configuration that is stealthier than any known $20$-dimensional lattice with the same particle density. This periodic configuration is a formal dual of Vardy's $20$-dimensional sphere packing.
\end{abstract}

\maketitle

\section{Introduction}

An arrangement of identical classical point particles in Euclidean space is called stealthy if it is transparent to radiation of all sufficiently long wavelengths \cite{FanPercusStillingerStillinger, UcheStillingerTorquato, UcheTorquatoStillinger, BattenStillingerTorquato, TorquatoZhangStillinger, Torquato}: short wavelengths can resolve its structure via scattering, while sufficiently long wavelengths pass without scattering. Every periodic arrangement has some degree of stealthiness, which is measured by the gap between the origin and the nearest Bragg peak in its diffraction pattern, while examples without periodicity include disordered ground states and stacked-slider phases \cite{ZhangStillingerTorquatoI, ZhangStillingerTorquatoII}.

If we fix the particle density (i.e., the number of particles per unit volume in space), then it is natural to ask which particle configuration is stealthiest, in the sense of maximizing the gap in the diffraction pattern, or equivalently minimizing the threshold wavelength beyond which there is no scattering. In this paper, we prove that the $E_8$ root lattice is the stealthiest particle arrangement in $\R^8$ and the Leech lattice is the stealthiest in $\R^{24}$, among a broad class of both crystalline and disordered competitors. In contrast, the conjectured optimizers in two and three dimensions are the triangular lattice and the body-centered cubic lattice \cite{TorquatoZhangStillinger}, but no proof is known. Our proof draws on the sharp auxiliary functions used to solve the sphere packing problem in eight and twenty-four dimensions \cite{Viazovska,CKMRV2017}.

Because the physical motivation for stealthy materials comes from ground states of interacting particle systems, we formulate our theorem in terms of isometry-invariant point processes on $\R^d$, following Ghosh and Lebowitz \cite{GhoshLebowitz} and Bj\"orklund and Byl\'ehn \cite{BjorklundBylehn}. This setting includes the case of periodic configurations, because averaging over isometries yields an isometry-invariant point process.\footnote{If $\cC$ is a periodic configuration, then the lattice $\Lambda$ of translation symmetries of $\cC$ is a subgroup of the group $G$ of isometries of $\R^d$, and $G/\Lambda$ has finite Haar measure. We can therefore obtain an isometry-invariant point process by averaging over $G/\Lambda$.} The point-process perspective is different from the deterministic approach often taken in discrete geometry, and it avoids complications from degenerate configurations. At the same time, we combine it with techniques from discrete geometry, such as linear programming bounds.

We use the point process definitions from the book of Last and Penrose \cite{LastPenrose}. We normalize the intensity of the point processes to be~$1$, and we require them to be locally square-integrable.

Given a point process $\eta$ as above, we will define its stealth radius $\stealth(\eta)$ in Section~\ref{section:background}. Under the Fourier transform convention customarily used in physics, $2\pi \stealth(\eta)$ is the lower edge of the scattering spectrum \cite{BattenStillingerTorquato,TorquatoZhangStillinger,MorseSteinhardtTorquato}. Stealthiness is often quantified using the parameter
\[
\chi = \frac{\vol\mathopen{}\left(B^d_{\stealth(\eta)}(0)\mathclose{}\right)}{2d},
\]
where $B^d_r(x)$ is the open ball of radius $r$ in $\R^d$ centered at $x$ (see equation~(35) in \cite{TorquatoZhangStillinger}). By this measure, $E_8$ achieves $\chi = \pi^4/24 = 4.05871212\dots$ and the Leech lattice achieves $\chi = 1024 \pi^{12}/1403325 = 674.43510368\dots.$ Both values are far above the threshold $\chi = 1/2$ below which counting degrees of freedom suggests that disordered stealthy arrangements are plentiful \cite{TorquatoZhangStillinger}.

Our main theorem gives sharp upper bounds in $\R^8$ and $\R^{24}$:

\begin{theorem}\label{theorem:main}
Let $\eta$ be an isometry-invariant, locally square-integrable point process on $\R^d$ with intensity~$1$. Then
\[
\stealth(\eta) \le \begin{cases}
\sqrt{2} & \text{if $d=8$, and}\\
2 & \text{if $d=24$.}
\end{cases}
\]
If $d=8$, then $\stealth(\eta)=\sqrt{2}$ if and only if $\eta$ is the isometry-invariant $E_8$ process; if $d=24$, then $\stealth(\eta)=2$ if and only if $\eta$ is the isometry-invariant Leech lattice process.
\end{theorem}

Uniqueness amounts to a form of crystallization: the only way for a point process to achieve maximal stealthiness in $\R^8$ or $\R^{24}$ is to become an isometric copy of $E_8$ or the Leech lattice, respectively.

Torquato, Zhang, and Stillinger \cite{TorquatoZhangStillinger} interpreted stealthiness as a sphere packing problem in Fourier space, and Morse, Steinhardt, and Torquato \cite{MorseSteinhardtTorquato} argued that the stealthiest configuration must be a lattice.\footnote{We use the terminology common in mathematics, where a lattice is a discrete subgroup of $\R^d$ with rank $d$ and a periodic configuration is the union of finitely many translates of a lattice. In many physics papers, periodic configurations are called lattices and lattices are called Bravais lattices.} That conclusion would imply our theorem, because the stealthiest lattice is the dual of the densest lattice packing of spheres.

However, there is a gap in the argument. It is based on a lemma from \cite{TorquatoZhangStillinger} regarding unions of stealthy configurations. This lemma assumes that all the constituent configurations are stealthy up to a fixed Fourier-space cutoff radius $K$, and it correctly computes the value of $\chi$ associated with that cutoff and the combined particle density. The union may nevertheless be stealthy beyond radius $K$, because its basis form factor can vanish at the shortest vectors of the dual of the underlying lattice. Thus equation~(14) in \cite{MorseSteinhardtTorquato} gives only a lower bound on the value of $\chi$ computed from the union's stealth radius, and their argument does not establish that the stealthiest configuration must be a lattice.

We conjecture that in all sufficiently high dimensions, the stealthiest configuration will not be a lattice. There is little evidence either way, but our conjecture is motivated by the intuition that lattices represent only a narrow and unrepresentative slice of all particle configurations.

In Section~\ref{section:vardy}, we construct a periodic configuration in $\R^{20}$ that is stealthier than any known $20$-dimensional lattice with the same particle density. To the best of our knowledge, this is the first such example in any dimension. It is formally dual to Vardy's $20$-dimensional sphere packing \cite{Vardy}, in the terminology of \cite{CKRS}; in other words, it is as if Vardy's packing were a lattice and therefore had a dual lattice. Vardy's packing is denser than every known lattice packing in $\R^{20}$ (although not every known packing \cite{CHLWT}), and its formal dual is therefore stealthier than any known lattice.

We cannot rule out the possibility that the stealthiest configuration is always a formal dual of the densest sphere packing, but that seems exceedingly unlikely. We see no reason why the densest sphere packing should even have a formal dual, and indeed the densest sphere packing known in $\R^{10}$ provably does not (Proposition~5.4 in \cite{CKRS}).

We will prove Theorem~\ref{theorem:main} in Section~\ref{section:proof}, after recalling some definitions and properties of point processes in Section~\ref{section:background}. Section~\ref{section:extensions} then explains how the same technique addresses uniqueness for related problems such as sphere packing and energy minimization. Finally, Section~\ref{section:vardy} constructs the formal dual of Vardy's packing.

\section{Point processes and pair correlations}
\label{section:background}

Before turning to the proof of Theorem~\ref{theorem:main}, we will review the pair-correlation formalism for stationary point processes. Let $\lambda_d$ denote Lebesgue measure on $\R^d$, and let $\delta_x$ denote a delta function at $x$, which we view both as a Dirac measure and as a tempered distribution. We use the normalization
\[
\widehat{f}(\xi) = \int_{\R^d} f(x) e^{-2\pi i \langle x, \xi \rangle} \, dx
\]
for the Fourier transform, where $\langle \cdot,\cdot \rangle$ denotes the usual inner product on $\R^d$, and we extend the Fourier transform in the usual way to tempered distributions via duality, so that $\widehat{\delta}_0 = 1$ and $\widehat{1} = \delta_0$.

We regard a point process $\eta$ as a random counting measure, with $\eta(B)$ being the number of particles located in $B$. We allow multiple particles to occupy the same location, and references to particles will always take these multiplicities into account. For example, when we count ordered pairs of particles below, we mean ordered pairs of particle labels, rather than merely ordered pairs of points at which particles are located.

We say a point process $\eta$ on $\R^d$ is \emph{stationary} if its law is invariant under translations, \emph{isometry invariant} if it is invariant under the full Euclidean isometry group, and \emph{locally square-integrable} if $\E[\eta(B)^2] < \infty$ for every bounded Borel set $B$ in $\R^d$. In principle we do not require general point processes to be locally finite (following \cite{LastPenrose}), but every locally square-integrable point process is locally finite, so we will not need to state this hypothesis.

Every locally finite point process $\eta$ on $\R^d$ is \emph{proper} by Corollary~6.5 in \cite{LastPenrose}, which means there are random elements $X_1,X_2,\dotsc \in \R^d$ and a random variable $k \in \Z_{\ge 0} \cup \{\infty\}$ such that
\[
\eta = \sum_{n=1}^k \delta_{X_n}.
\]
This representation allows multiplicities, because $X_1,X_2,\dots$ need not be distinct.

To understand pair correlations, we will need to analyze pairs of particles in samples from $\eta$. The \emph{second factorial measure} $\eta^{(2)}$ of $\eta$ counts ordered pairs of distinct particles (i.e., $\eta^{(2)}(B)$ is the number of ordered pairs of distinct particles in $B \subseteq \R^d \times \R^d$), and the \emph{second factorial moment measure} $\alpha_{2,\eta}$ is its expectation $\E[\eta^{(2)}]$. When $\eta$ is stationary, it has a \emph{reduced second factorial moment measure} $\alpha^!_{2,\eta}$. This measure is locally finite if and only if $\eta$ is locally square-integrable, in which case it is characterized by
\begin{equation} \label{eq:reducedfactorialmoment}
\int f(x,y) \, d\alpha_{2,\eta}(x,y) = \iint f(x,x+y) \, d\alpha^!_{2,\eta}(y) \, dx
\end{equation}
for nonnegative Borel-measurable functions $f$ on $\R^d \times \R^d$ (Proposition~8.7 in \cite{LastPenrose}). In other words, $\alpha^!_{2,\eta}$ is the intensity measure of relative displacements of ordered pairs of distinct particles.

We henceforth assume that $\eta$ is stationary and locally square-integrable. Then $\alpha^!_{2,\eta}$ is translation bounded: for every compact set $K$ in $\R^d$,
\[
\sup_{t \in \R^d} \alpha^!_{2,\eta}(K + t) < \infty.
\]
Indeed, \eqref{eq:reducedfactorialmoment} implies that
\[
\alpha^!_{2,\eta}(K + t) \le \alpha_{2,\eta}([0,1]^d \times ([0,1]^d + K + t)).
\]
We have
\[\eta^{(2)}([0,1]^d \times ([0,1]^d + K + t)) \le \eta([0,1]^d) \eta([0,1]^d + K + t),
\]
and therefore
\begin{align*}
\alpha^!_{2,\eta}(K + t) &\le \E[\eta([0,1]^d) \eta([0,1]^d + K + t)]\\
&\le \sqrt{\E[\eta([0,1]^d)^2]} \sqrt{\E[\eta([0,1]^d + K + t)^2]}
\end{align*}
by the Cauchy--Schwarz inequality. This upper bound is independent of $t$ by stationarity and finite by local square-integrability, which proves translation boundedness. As a consequence of translation boundedness, $\alpha^!_{2,\eta}$ defines a tempered distribution, and $\alpha^!_{2,\eta}(B^d_r(0)) = O(r^d)$ as $r \to \infty$.

Suppose now that $\eta$ is not only stationary and locally square-integrable, but also has \emph{intensity}~$1$ (i.e., $\E[\eta(B)] = \lambda_d(B)$ for all bounded Borel sets $B$ in $\R^d$). Then the signed Radon measure
\[
\kappa_\eta = \delta_0 + \alpha^!_{2,\eta} - \lambda_d
\]
is called the \emph{reduced covariance measure} of $\eta$. It is an even measure (i.e., $\kappa_\eta(-B) = \kappa_\eta(B)$) and defines a tempered distribution. We can also identify the $\lambda_d$ term with the tempered distribution $1$.

For a Schwartz function $\varphi \in \cS(\R^d)$, write
\[
\eta(\varphi) = \int \varphi(x) \, d\eta(x)
\quad\text{and}\quad
\widetilde\varphi(x) = \overline{\varphi(-x)}.
\]
Then $\E[|\eta(\varphi)|^2] < \infty$ for every $\varphi \in \cS(\R^d)$: we can tile space by cubes, bound $|\varphi|$ by its maximum on each cube, and use the local square-integrability and stationarity of $\eta$ together with Minkowski's inequality.

For all $\varphi,\psi \in \cS(\R^d)$, the reduced covariance measure satisfies
\begin{equation}
\label{eq:covariance}
\begin{split}
\Cov[\eta(\varphi), \eta(\psi)] &:=
\E\left[\left(\eta(\varphi)-\int\varphi(x) \, dx\right)
\overline{\left(\eta(\psi)-\int\psi(x) \, dx\right)}\right]\\
&= \left\langle\kappa_\eta,\varphi*\widetilde{\psi}\right\rangle.
\end{split}
\end{equation}
In particular, $\kappa_\eta$ is positive definite, because
\[
\left\langle \kappa_\eta, \varphi * \widetilde{\varphi} \right\rangle = \Var[ \eta(\varphi)] \ge 0.
\]
For completeness, we will verify the identity \eqref{eq:covariance}. We begin with
\begin{align*}
\E\left[\eta(\varphi) \overline{\eta(\psi)}\right] &= \iint \varphi(x) \overline{\psi(x+y)} \, dx \, d\alpha^!_{2,\eta}(y)
+
\int \varphi(x) \overline{\psi(x)} \, dx\\
&= \int \left(\varphi * \widetilde{\psi}\,\right)(-y) \, d\alpha^!_{2,\eta}(y) 
+
\int \varphi(x) \overline{\psi(x)} \, dx,
\end{align*}
which follows from Proposition~8.7 in \cite{LastPenrose} (see the paragraph after the proof of the proposition). Combining this equation with
\[
\int \left(\varphi * \widetilde{\psi}\,\right)(-y) \, dy = \left(\int \varphi(x) \, dx\right) \overline{\left(\int \psi(x) \, dx\right)}, 
\]
\[
\left(\varphi * \widetilde{\psi}\,\right)(0) = \int \varphi(x) \overline{\psi(x)} \, dx, 
\]
and the symmetry of $\kappa_\eta$ under $x \mapsto -x$ completes the proof.

By the Bochner--Schwartz theorem, the Fourier transform of a positive-definite tempered distribution is a positive tempered measure. We define the \emph{structure factor} $S_\eta$ of $\eta$ to be the Fourier transform $\widehat{\kappa}_\eta$ of the reduced covariance measure. Because $\kappa_\eta$ is even, $S_\eta$ is even as well, and Fourier inversion gives $\kappa_\eta = \widehat{S}_\eta$. The structure factor is also known as the Bartlett spectral measure (see, for example, Section~8.2 of \cite{DaleyVereJones}).

In terms of $S_\eta$,
\[
\Cov[\eta(\varphi),\eta(\psi)] = \int \widehat{\varphi}(\xi) \overline{\widehat{\psi}(\xi)} \, dS_\eta(\xi).
\]
As a consequence, if $S_\eta(\supp \widehat{\varphi}) = 0$, then $\eta(\varphi)$ has variance $0$ and is therefore almost surely equal to its expected value $\int \varphi(x) \, dx$.

When $\alpha^!_{2,\eta}$ is absolutely continuous with respect to $\lambda_d$ and $\eta$ has intensity~$1$, physics notation writes $\alpha^!_{2,\eta} = g_2 \lambda_d$ and calls $g_2$ the \emph{pair correlation function}. More generally, we use $g_2$ to denote the tempered distribution defined by $\alpha^!_{2,\eta}$. With this convention, the structure factor $S_\eta = 1+\widehat{g_2-1}$ describes the scattering intensity for incident radiation. We can therefore use it to define the stealth radius:

\begin{definition}
Let $\eta$ be a stationary, locally square-integrable point process on $\R^d$ with intensity~$1$. The \emph{stealth radius} of $\eta$ is
\[
\stealth(\eta) = \sup \,\{r \ge 0 : S_\eta(B^d_r(0))=0\},
\]
where $B^d_r(0)$ is the open ball of radius $r$ centered at the origin in $\R^d$.
\end{definition}

The stealth radius is always finite: otherwise $S_\eta=0$ and therefore $\kappa_\eta=0$, which would contradict $\kappa_\eta(\{0\}) = 1 + \alpha^!_{2,\eta}(\{0\}) \ge 1$. Note also that if $r = \stealth(\eta)$, then $S_\eta(B^d_r(0))=0$ follows from $B^d_r(0) = \bigcup_{s<r} B^d_s(0)$ using continuity from below.

Given a lattice $\Lambda$ in $\R^d$, we can obtain a stationary point process $\eta$ by averaging over translation by vectors modulo $\Lambda$. If $\Lambda$ has particle density~$1$, then
\[
\alpha^!_{2,\eta} = \sum_{x\in\Lambda\setminus\{0\}}\delta_x,
\]
and
\[
S_\eta = \sum_{\xi\in\Lambda^*\setminus\{0\}}\delta_\xi,
\]
where
\[
\Lambda^* = \{ \xi \in \R^d : \text{$\langle \xi,x \rangle \in \Z$ for all $x \in \Lambda$} \}
\]
is the dual lattice; the first identity follows from the definition, and the second from Poisson summation. The stealth radius of $\eta$ is therefore the shortest nonzero vector length of $\Lambda^*$.

More generally, if a periodic configuration $\cC$ with particle density~$1$ is the union of pairwise disjoint translates $\Lambda+v_j$ of $\Lambda$ with $1 \le j \le n$, then the corresponding stationary process $\eta$ satisfies
\[
\alpha^!_{2,\eta} = \frac{1}{n} \sum_{j,k=1}^n \sum_{\substack{x \in \Lambda\\ x+v_j-v_k \ne 0}}\delta_{x+v_j-v_k}
\]
and
\[
S_\eta = \sum_{\xi\in\Lambda^*\setminus\{0\}}
\left|\frac{1}{n}\sum_{j=1}^n e^{2\pi i \langle v_j, \xi \rangle}\right|^2\delta_\xi.
\]
The stealth radius is the smallest $|\xi|$ with $\xi \in \Lambda^*\setminus\{0\}$ and
\[
\sum_{j=1}^n e^{2\pi i \langle v_j, \xi \rangle} \ne 0.
\]

The \emph{isometry-invariant $\cC$ process} is obtained from the stationary process by averaging over the action of $O(d)$. It has the same stealth radius as the stationary process. For example, the isometry-invariant $E_8$ and Leech lattice processes have stealth radii $\sqrt{2}$ and $2$, respectively. 

\section{Proof of the main theorem}
\label{section:proof}

The proof of our upper bound for the stealth radius is obtained by interchanging the roles of physical space and Fourier space in the Cohn--Elkies linear programming bound for the sphere packing density \cite{CohnElkies}. Given an auxiliary function satisfying certain inequalities, we obtain a bound:

\begin{theorem} \label{theorem:LP}
Let $f \in \cS(\R^d)$ be real-valued and even, and let $r>0$. If $f(0) = \widehat{f}(0) = 1$, $f(x) \le 0$ whenever $|x| \ge r$, and $\widehat{f}(\xi) \ge 0$ for all $\xi$, then every stationary, locally square-integrable point process $\eta$ on $\R^d$ with intensity~$1$ satisfies
\[
\stealth(\eta) \le r.
\]
Moreover, $\stealth(\eta)=r$ if and only if all three of the following conditions hold: $S_\eta(B^d_r(0))=0$, $\supp \alpha^!_{2,\eta} \subseteq \{ x \in \R^d: \widehat{f}(x)=0\}$, and $\supp S_\eta \subseteq \{ \xi \in \R^d: f(\xi)=0\}$.
\end{theorem}

In the equality conditions, the use of $\widehat{f}$ to describe $\supp \alpha^!_{2,\eta}$ and $f$ to describe $\supp S_\eta$ is intentional, and is the reverse of the relationship in Section~5 of \cite{CohnElkies}.

\begin{proof}
Let $\eta$ be such a point process, and suppose $S_\eta(B^d_R(0))=0$ with $R>0$. We begin by defining $g$ by $g(x) = R^d \, \widehat{f}(xR/r)$, or equivalently $\widehat{g}(\xi) = r^d f(\xi r/R)$ because $f$ is even. This rescaling ensures that $g(x) \ge 0$ for all $x$ and $\widehat{g}(\xi) \le 0$ for $|\xi| \ge R$. Then
\begin{align*}
\int \widehat{g}(\xi) \, dS_\eta(\xi) &= \int g(x) \, d\kappa_\eta(x)\\
&= g(0) + \int g(x) \, d\alpha^!_{2,\eta}(x) - \int g(x) \, dx\\
&\ge g(0) + 0 - \widehat{g}(0)\\
&= R^d - r^d.
\end{align*}
Furthermore, $\widehat{g}$ is nonpositive on $\supp S_\eta$, and so we conclude that $0 \ge R^d - r^d$ and therefore $r \ge R$. Taking the supremum over all such $R$ yields $\stealth(\eta) \le r$.

Setting $R=r$ shows that equality holds if and only if $S_\eta(B^d_r(0))=0$ and
\[
 \int g(x) \, d\alpha^!_{2,\eta}(x) = \int \widehat{g}(\xi) \, dS_\eta(\xi) = 0.
\]
Because $g = r^d \widehat{f}$ when $r=R$, the two vanishing integral conditions are equivalent to $\supp \alpha^!_{2,\eta} \subseteq \{ x \in \R^d: \widehat{f}(x)=0\}$ and $\supp S_\eta \subseteq \{ \xi \in \R^d: f(\xi)=0\}$.
\end{proof}

The inequality constraints in Theorem~\ref{theorem:LP} are identical to those in Theorem~3.2 of \cite{CohnElkies}, and thus linear programming bounds for sphere packing yield bounds for the stealth radius as well. For example, we obtain an asymptotic bound using the auxiliary functions constructed by OpenAI \cite[Theorem~4.1 of Chapter~1]{OpenAI}:

\begin{proposition}
As $d \to \infty$, every stationary, locally square-integrable point process of intensity~$1$ in $\R^d$ has stealth radius at most $(1/\pi+o(1))\sqrt{d}$.
\end{proposition}

In addition to Theorem~\ref{theorem:LP}, we will need the following characterization of the $E_8$ and Leech lattice processes.

\begin{lemma} \label{lemma:rigidity}
Let $d \in \{8,24\}$, let $\Lambda_d$ be $E_8$ or the Leech lattice, respectively, and let 
\[
D = \begin{cases} 2 \Z_{>0} & \text{if $d=8$, and}\\
2 \Z_{>1} & \text{if $d=24$.}
\end{cases}
\]
If $\eta$ is an isometry-invariant, locally square-integrable point process on $\R^d$ with intensity~$1$ such that
\[
\supp \alpha^!_{2,\eta} \subseteq \{x \in \R^d : |x|^2 \in D\},
\]
then $\eta$ is the isometry-invariant $\Lambda_d$ process.
\end{lemma}

\begin{proof}
Applying \eqref{eq:reducedfactorialmoment} to
\[
(x,y) \mapsto \begin{cases}
1 & \text{if $|x-y|^2 \notin D$, and}\\
0 & \text{otherwise}
\end{cases}
\]
shows that the total number of ordered pairs of distinct particles $x,y$ with $|x-y|^2 \notin D$ has expectation zero. Therefore almost surely all distinct particles $x,y$ satisfy $|x-y|^2 \in D$. In particular, because $0 \notin D$, the point process is simple: no two particles ever occupy the same location.

Let $X$ be any nonempty subset of $\R^d$ satisfying this distance constraint, and let $x_0 \in X$. As shown in Lemma~8.2 of \cite{CohnElkies} using the polarization identity, the $\Z$-span of the translate $X-x_0$ must be an even lattice within its real span, i.e., a lattice in which all squared norms are even integers and therefore all inner products are integers by polarization. 

Given a positive integer $n$, let
\[
Y_n = \frac{\eta(B^d_n(0))}{\vol (B^d_n(0))},
\]
which is a random variable with expected value~$1$. Then $Y_n$ is almost surely bounded by a constant that does not depend on $n$, because the particles are almost surely separated by distance at least $\sqrt{2}$. Specifically, the open balls of radius $\sqrt{2}/2$ centered at the particles in $B^d_n(0)$ are disjoint and contained in $B^d_{n+\sqrt{2}/2}(0)$, and therefore
\begin{equation} \label{eq:packing}
Y_n \le \frac{\vol \mathopen{}\left(B^d_{n+\sqrt{2}/2}(0)\right)\mathclose{}}{\vol (B^d_n(0)) \vol \mathopen{}\left(B^d_{\sqrt{2}/2}(0)\right)\mathclose{}}
= \frac{\left(1+\frac{\sqrt{2}}{2n}\right)^d}{\vol \mathopen{}\left(B^d_{\sqrt{2}/2}(0)\right)\mathclose{}}
\le \frac{\left(1+\frac{\sqrt{2}}{2}\right)^d}{\vol \mathopen{}\left(B^d_{\sqrt{2}/2}(0)\right)\mathclose{}}.
\end{equation}

For an empty sample from $\eta$, $Y_n=0$ for all $n$. For every nonempty sample satisfying the distance constraint, the points are contained in a translate of an even lattice $\Lambda_0$, possibly of lower dimension than $d$. If $\Lambda_0$ is full dimensional and $G$ is the Gram matrix of a basis, then
\[
\limsup_{n \to \infty} Y_n \le \frac{1}{\sqrt{\det G}},
\]
while this limit superior is zero if $\Lambda_0$ is not full dimensional. Because $\det G$ is a positive integer, in either case
\[
\limsup_{n \to \infty} Y_n \le 1.
\]

Because there is a constant $C$ such that almost surely $Y_n \le C$ for every $n$, Fatou's lemma applied to the nonnegative random variables $C-Y_n$ gives
\[
1 = \limsup_{n \to \infty} \E[Y_n] \le \E\left[\limsup_{n \to \infty} Y_n\right] \le 1,
\]
and we conclude that
\[
\limsup_{n \to \infty} Y_n = 1
\]
almost surely. It follows that $\eta$ is almost surely nonempty and that $\Lambda_0$ is almost surely full dimensional and unimodular.

Let $\Lambda$ be the random translated lattice constructed as follows. Start with a proper enumeration $X_1,X_2,\dots$ of the particles in $\eta$, let $\Lambda_0$ be the $\Z$-span of the differences $X_n-X_1$, and set $\Lambda = \Lambda_0 + X_1$. The resulting coset is independent of the choice of enumeration and base point, and enumerating its elements without repetitions establishes measurability. This construction is equivariant under isometries, and thus $\Lambda$ is an isometry-invariant point process. The points in $\Lambda$ are separated by distance at least $\sqrt{2}$, so $\Lambda$ is locally square-integrable. Furthermore,
\[
\lim_{n \to \infty} \frac{\Lambda(B^d_n(0))}{\vol(B^d_n(0))} = 1
\]
almost surely because $\Lambda_0$ is unimodular. The random variables $\Lambda(B^d_n(0))/\vol(B^d_n(0))$ have the same deterministic upper bound as $Y_n$ in \eqref{eq:packing}, and their expectations are equal to the intensity of $\Lambda$. Dominated convergence therefore shows that $\Lambda$ has intensity~$1$.

The points in $\Lambda$ that are missing from $\eta$ form an isometry-invariant point process of intensity~$0$, because both $\eta$ and $\Lambda$ have intensity~$1$. A point process of intensity~$0$ is almost surely empty, and therefore $\eta = \Lambda$ almost surely. Finally, $\eta$ must almost surely be isometric to $\Lambda_d$, since up to isometry $E_8$ is the only even unimodular lattice in $\R^8$ and the Leech lattice is the only even unimodular lattice in $\R^{24}$ that has no vectors of length $\sqrt{2}$ (see, for example, Proposition~2.5 and Theorem~4.1 in \cite{Ebeling}). We conclude that $\eta$ must be the isometry-invariant $\Lambda_d$ process, by the uniqueness of the invariant probability measure on the corresponding homogeneous space.
\end{proof}

We now have all the tools needed to prove our main theorem.

\begin{proof}[Proof of Theorem~\ref{theorem:main}]
The upper bounds for $\stealth(\eta)$ follow from Theorem~\ref{theorem:LP} using the auxiliary functions constructed by Viazovska \cite{Viazovska} to solve the sphere packing problem in $\R^8$ and by Cohn, Kumar, Miller, Radchenko, and Viazovska \cite{CKMRV2017} to solve it in $\R^{24}$. These bounds are attained by the isometry-invariant $E_8$ and Leech lattice processes, so all that remains is to determine the equality cases.

Let $d \in \{8,24\}$, let $\Lambda_d$ be $E_8$ or the Leech lattice, respectively, and let $f_d$ be the optimal auxiliary function. The lengths of the nonzero vectors in $\Lambda_d$ are $\{\sqrt{2k} : k \in \Z_{>0}\}$ when $d=8$ and $\{\sqrt{2k} : k \in \Z_{>1}\}$ when $d=24$.

To understand the equality cases using Theorem~\ref{theorem:LP}, we will need to know that $\widehat{f}_d(x)=0$ exactly when $|x|$ is a nonzero vector length in $\Lambda_d$. In dimension~$8$, this assertion is part of Theorem~3 in \cite{Viazovska}. In dimension~$24$, equation~(4.2) and Lemma~A.2 in \cite{CKMRV2017} handle the case $|x| > \sqrt{2}$, Lemma~A.3 and the decomposition before equation~(4.5) handle $|x| < \sqrt{2}$, and $\widehat{f}_{24}(x) = 1/156$ when $|x| = \sqrt{2}$.

Suppose $\eta$ attains the upper bound. Theorem~\ref{theorem:LP} then gives
\[
\supp \alpha^!_{2,\eta} \subseteq \{ x \in \R^d: \widehat{f}_d(x)=0\},
\]
which is precisely the support condition needed for Lemma~\ref{lemma:rigidity}. The lemma therefore implies that $\eta$ is the isometry-invariant $\Lambda_d$ process, as desired.
\end{proof}

Note that the proof of the upper bound for $\stealth(\eta)$ requires only translation invariance, not full isometry invariance. The stationary processes that attain this bound are precisely the mixtures of the stationary lattice processes obtained from different orientations of the optimal lattice.

\section{Extensions to related problems}
\label{section:extensions}

\subsection{Sphere packing}

The uniqueness argument from the previous section also applies to related optimization problems. For example, the Bowen--Radin framework \cite{Bowen2000,Bowen2003a,Bowen2003b,BowenHoltonRadinSadun,BowenRadin2004} studies isometry-invariant probability measures on the space of packings. Equivalently, for packings of congruent spheres, it deals with random packings whose sphere centers form an isometry-invariant point process. Wackenhuth \cite{Wackenhuth} pioneered the use of linear programming bounds in this setting, while our methods prove uniqueness.

The \emph{minimal distance} of a point process $\eta$ in $\R^d$ is the supremum of all $r \ge 0$ such that almost surely all distinct particles are separated by distance at least $r$.

\begin{theorem}
Let $d \in \{8,24\}$, let $\Lambda_d$ be $E_8$ or the Leech lattice, respectively, and let $r_d$ be the minimal distance of $\Lambda_d$, so that $r_8=\sqrt{2}$ and $r_{24}=2$. Suppose $\eta$ is an isometry-invariant point process on $\R^d$ with minimal distance at least $r_d$. Then the intensity of $\eta$ is at most~$1$, with equality if and only if $\eta$ is the isometry-invariant $\Lambda_d$ process.

Equivalently, if $\eta$ is an isometry-invariant point process of intensity~$1$ and minimal distance $r$, then $r \le r_d$, with equality if and only if $\eta$ is the isometry-invariant $\Lambda_d$ process.
\end{theorem}

\begin{proof}
The minimal distance condition implies local square-integrability, and so we can apply the covariance and structure factor formalism from Section~\ref{section:background}. First suppose $\eta$ has intensity~$1$ and minimal distance at least $r_d$, and let $f_d$ be the optimal auxiliary function for the linear programming bound. Then
\[
0 \le \int \widehat{f}_d(\xi) \, dS_\eta(\xi)
= \int f_d(x) \, d\kappa_\eta(x)
= \int f_d(x) \, d\alpha^!_{2,\eta}(x) \le 0,
\]
where the first inequality uses $\widehat{f}_d(\xi) \ge 0$, the second equality uses $f_d(0) = \widehat{f}_d(0) = 1$, and the final inequality uses $f_d(x) \le 0$ for $|x| \ge r_d$. Thus, equality holds throughout. 

Because $f_d$ is continuous and nonpositive on the support of $\alpha^!_{2,\eta}$, the equality above and the minimal-distance condition give
\[
\supp\alpha^!_{2,\eta} \subseteq \{ x \in \R^d : \text{$f_d(x) = 0$ and $|x| \ge r_d$}\}.
\]
By Theorem~3 in \cite{Viazovska} and Section~4 of \cite{CKMRV2017}, this set is precisely the spheres centered at the origin that contain nonzero points of $\Lambda_d$. Lemma~\ref{lemma:rigidity} therefore implies that $\eta$ is the isometry-invariant $\Lambda_d$ process, and the assertion for arbitrary intensity follows by rescaling.
\end{proof}

Optimal sphere packing processes in other dimensions need not be unique. For example, in $\R^3$ the face-centered cubic and hexagonal close packing configurations give different optimal packing processes \cite{Hales,Flyspeck}.

\subsection{Energy minimization}

Similarly, ground states of isotropic pair potentials can be studied among isometry-invariant, locally square-integrable point processes. Cohn, Kumar, Miller, Radchenko, and Viazovska \cite{CKMRV2022} proved that $E_8$ and the Leech lattice are universally optimal and are unique ground states among periodic configurations, while our methods can be adapted as follows to prove uniqueness among point processes.

Our potential function will be a measurable function $p \colon [0,\infty) \to [0,\infty]$. If $\eta$ is a stationary, locally square-integrable point process on $\R^d$ with intensity~$1$, we define the \emph{energy} of $\eta$ under $p$ to be
\[
E_p(\eta) = \int p(|x|) \, d\alpha^!_{2,\eta}(x),
\]
which agrees with the definition in \cite{CKMRV2022} when $\eta$ is the isometry-invariant process derived from a unit-density periodic configuration. We will also write $E_p(\cC)$ for a unit-density periodic configuration $\cC$, in which case we mean to use the corresponding isometry-invariant point process. In physics terms, the average energy per particle under the potential function~$p$ is $E_p(\eta)/2$, because $\alpha^!_{2,\eta}$ counts ordered pairs.

The energy $E_p(\eta)$ may be infinite in general, but when $p(r) = e^{-\alpha r^2}$ with $\alpha>0$, it is always finite, because $\alpha^!_{2,\eta}(B^d_r(0)) = O(r^d)$ as $r \to \infty$. Furthermore, the Gaussian energy is a continuous function of $\alpha \in (0,\infty)$, by dominated convergence.

The analogue of linear programming bounds \cite{CohnKumar,CohndeCourcyIreland} is as follows:

\begin{theorem} \label{theorem:potential}
Let $f \in \cS(\R^d)$ be real-valued and even. If $f(x) \le p(|x|)$ for all $x$ and $\widehat{f}(\xi) \ge 0$ for all $\xi$, then every stationary, locally square-integrable point process $\eta$ on $\R^d$ with intensity~$1$ satisfies
\[
E_p(\eta) \ge \widehat{f}(0)-f(0).
\]
Equality holds for $\eta$ if and only if $f(x) = p(|x|)$ holds $\alpha^!_{2,\eta}$-almost everywhere and $\supp S_\eta \subseteq \{ \xi \in \R^d: \widehat{f}(\xi)=0\}$; if $p$ is finite-valued and continuous, the first condition becomes $\supp \alpha^!_{2,\eta} \subseteq \{ x \in \R^d: f(x)=p(|x|)\}$.
\end{theorem}

\begin{proof}
The identity $\widehat{\kappa}_\eta = S_\eta$ and the definition of $E_p(\eta)$ imply that
\[
E_p(\eta) - \big(\widehat{f}(0)-f(0)\big) = \int \big(p(|x|)-f(x)\big) \, d\alpha^!_{2,\eta}(x) + \int \widehat{f}(\xi) \, dS_\eta(\xi).
\]
Both terms on the right are nonnegative because $p(|x|) \ge f(x)$ and $\widehat{f}(\xi) \ge 0$, which proves the bound and the equality conditions.
\end{proof}

For $E_8$ and the Leech lattice, \cite{CKMRV2022} analyzed potential functions that are completely monotonic functions of squared distance. In other words, $p(r) = q(r^2)$ for $r>0$, where $q \colon (0,\infty) \to [0,\infty)$ is a smooth function satisfying $(-1)^k q^{(k)}(t) \ge 0$ for all $k \ge 0$, and we set $p(0) := \lim_{r \to 0+} p(r) \in[0,\infty]$. Inverse power laws and Gaussians are important examples, and Bernstein's theorem (Theorem~9.16 in \cite{Simon}) says that any such potential function can be written as a convergent integral 
\begin{equation}
\label{eq:bernstein}
p(r) = \int_{[0,\infty)} e^{-\alpha r^2} \, d\mu(\alpha)
\end{equation}
for $r>0$, where $\mu$ is a Borel measure on $[0,\infty)$.

\begin{theorem}
Let $d \in \{8,24\}$, let $\Lambda_d$ be $E_8$ or the Leech lattice, respectively, and let $p \colon [0,\infty) \to [0,\infty]$ be a completely monotonic function of squared distance. Suppose $\eta$ is an isometry-invariant, locally square-integrable point process on $\R^d$ with intensity $1$. Then $E_p(\eta) \ge E_p(\Lambda_d)$. If equality holds, $p$ is not identically zero, and $E_p(\Lambda_d) < \infty$, then $\eta$ must be the isometry-invariant $\Lambda_d$ process.
\end{theorem}

\begin{proof}
When $p$ is a Gaussian, Lemma~6.2 in \cite{CKMRV2022} provides an auxiliary function $f$ for which Theorem~\ref{theorem:potential} proves that $E_p(\eta) \ge E_p(\Lambda_d)$. When equality holds, Theorem~\ref{theorem:potential} gives
\[
\supp \alpha^!_{2,\eta} \subseteq \{ x \in \R^d: f(x)=p(|x|)\}.
\]
Lemma~6.2 in \cite{CKMRV2022} and the remark immediately following it state that $f(x) = p(|x|)$ precisely when $|x|$ is the length of a vector in $\Lambda_d \setminus \{0\}$, so Lemma~\ref{lemma:rigidity} proves uniqueness.

However, other potential functions may decay too slowly to be interpolated by a Schwartz function. For more general $p$, let $\mu$ satisfy \eqref{eq:bernstein}, and set
\[
G_\eta(\alpha) = E_{r \mapsto e^{-\alpha r^2}}(\eta) = \int e^{-\alpha |x|^2} \, d\alpha^!_{2,\eta}(x). 
\]
Then Tonelli's theorem implies that
\[
E_p(\eta) = \int_{[0,\infty)} G_\eta(\alpha) \, d\mu(\alpha).
\]
Gaussian optimality yields $G_\eta(\alpha) \ge G_{\Lambda_d}(\alpha)$ for all $\alpha>0$, and $G_\eta(0) = G_{\Lambda_d}(0)=\infty$, so integration gives
\[
E_p(\eta) \ge E_p(\Lambda_d).
\]

For uniqueness, suppose $p$ is not identically zero, $E_p(\Lambda_d)<\infty$, and $\eta$ is not the isometry-invariant $\Lambda_d$ process. Then uniqueness in the Gaussian case gives
\[
h(\alpha) := G_\eta(\alpha) - G_{\Lambda_d}(\alpha) > 0
\]
for $\alpha > 0$, and $\mu(\{0\})=0$ since otherwise $p(r) \ge \mu(\{0\}) > 0$ would imply $E_{p}(\Lambda_d)=\infty$. Because $\mu \ne 0$ and $\mu(\{0\})=0$, some compact interval $[a,b] \subseteq (0,\infty)$ has $\mu([a,b])>0$. By continuity,
\[
\min_{\alpha \in [a,b]} h(\alpha) > 0,
\]
and therefore
\begin{align*}
E_p(\eta) &= E_p(\Lambda_d) + \int_{(0,\infty)} h(\alpha) \, d\mu(\alpha)\\
&\ge E_p(\Lambda_d) + \mu([a,b]) \min_{\alpha \in [a,b]} h(\alpha)\\
&> E_p(\Lambda_d),
\end{align*}
as desired.
\end{proof}

\section{A stealthy formal dual of Vardy's packing}
\label{section:vardy}

\subsection{Formal duality and diffraction}

Let $\Lambda$ be a lattice in $\R^d$ and $\cP = \bigcup_{j=1}^n (\Lambda + v_j)$ a periodic configuration, where we assume the translates $\Lambda+v_j$ are distinct. The particle density of $\cP$ is $\density(\cP) = n/\vol(\R^d/\Lambda)$. Given a Schwartz function $f \in \cS(\R^d)$, we let
\[
\Sigma_f(\cP) = \frac{1}{n} \sum_{j,k=1}^n \sum_{x \in \Lambda} f(x+v_j-v_k)
\]
be the \emph{average pair sum} of $f$ over $\cP$, or equivalently the average of $\sum_{y \in \cP} f(y-x)$ over $x \in \cP$. Periodic configurations $\cP$ and $\cQ$ are \emph{formally dual} if
\[
\Sigma_f(\cP) = \density(\cP) \, \Sigma_{\widehat{f}}(\cQ)
\]
for all $f \in \cS(\R^d)$. See \cite{CKRS} and \cite{LiPottSchuler} for the theory of formal duality.

In our setting, we can rephrase formal duality as follows. Suppose $\cP$ and $\cQ$ have particle density~$1$, and $\eta_\cP$ and $\eta_\cQ$ are the stationary point processes obtained by random translation. If we set
\[
\mu_\cP=\delta_0+\alpha^!_{2,\eta_\cP}
\quad\text{and}\quad
\mu_\cQ=\delta_0+\alpha^!_{2,\eta_\cQ},
\]
then $\Sigma_f(\cP) = \int f \, d\mu_\cP$ and similarly for
$\cQ$, so $\cP$ and $\cQ$ are formally dual if and only if
$\mu_\cP=\widehat{\mu}_\cQ$. Equivalently,
\[
S_{\eta_\cQ}=\alpha^!_{2,\eta_\cP}
\quad\text{and}\quad
S_{\eta_\cP}=\alpha^!_{2,\eta_\cQ}.
\]
Thus formal duality interchanges the minimal distance of one configuration with the stealth radius of the other.

More generally, without the unit-density normalization, let $\nu_{\cP}$ be the measure characterized by $\Sigma_f(\cP) = \int f \, d\nu_\cP$ for all $f \in \cS(\R^d)$. Formal duality is then equivalent to $\nu_{\cP} = \density(\cP) \, \widehat{\nu}_{\cQ}$. The density factor changes the weights but not the support: the minimal distance of $\cP$ is the distance from the origin to the first nonzero Bragg peak of $\cQ$, and vice versa.

\subsection{Vardy's sphere packing}

We now construct a formal dual of Vardy's $20$-dimensional sphere packing \cite{Vardy}. We use the standard coordinate model of the Leech lattice $\Lambda_{24}$, consisting of certain integer vectors divided by $\sqrt{8}$ (see, for example, \cite[p.~131]{SPLAG}). Let $U$ be the span of the first four coordinate vectors, let $V = U^\perp$, let $\pi_U$ and $\pi_V$ be the orthogonal projections, and let
\[
K = \Lambda_{24} \cap U \quad\text{and}\quad L = \Lambda_{24} \cap V.
\]
In fact,
\[
K = \sqrt{2} D_4 = \sqrt{2} \{x \in \Z^4 : x_1+x_2+x_3+x_4 \equiv 0 \pmod{2} \},
\]
while $L$ is the laminated lattice $\Lambda_{20}$.

Because $\Lambda_{24}$ is unimodular and $K$ and $L$ are primitive orthogonal complements, there is a gluing isomorphism $\gamma \colon K^*/K \to L^*/L$ characterized by $\gamma(u+K) = v+L$ if and only if $u+v \in \Lambda_{24}$ (see \cite[p.~100]{SPLAG}). If $\gamma(u_1+K) = v_1+L$ and $\gamma(u_2+K)=v_2+L$, then $\langle u_1, u_2 \rangle \equiv - \langle v_1, v_2 \rangle \pmod{\Z}$, because $\langle u_1, u_2 \rangle + \langle v_1, v_2 \rangle = \langle u_1+v_1, u_2+v_2 \rangle \in \Z$.

Conway and Sloane \cite{ConwaySloane1996} described Vardy's packing using the four points
\[
A = \frac{1}{\sqrt{8}} \{ (0,0,0,0),\ (1,1,1,1),\ (2,0,0,0),\ (1,1,1,-1)\} \subseteq K^*.
\]
The sphere centers in Vardy's packing are
\[
\cV_{20} = \{\pi_V(w) : w \in \Lambda_{24} \text{ and } \pi_U(w) \in A\}.
\]
The set $A$ is a regular tetrahedron with squared edge length $1/2$, and the squared minimal distance in $\Lambda_{24}$ is $4$, which means $\cV_{20}$ has squared minimal distance at least $4-1/2 = 7/2$. In fact it is exactly $7/2$, because $\cV_{20}$ contains the origin and $(-3,1,1,\dots,1)/\sqrt{8}$, which is the projection of
\[
w = \frac{1}{\sqrt{8}}(1,1,1,1,-3,\overbrace{1,1,\dots,1}^{\text{nineteen $1$'s}})
\]
from $\Lambda_{24}$ to $V$ (see equation~(133) in Chapter~4 of \cite{SPLAG} for why $w \in \Lambda_{24}$).

\subsection{The Li--Pott--Sch\"uler dual pair}

Let
\[
e_1 = \frac{1}{\sqrt{8}}(-2,0,0,2), \quad e_2 = \frac{1}{\sqrt{8}}(1,1,1,-1), \quad e_3 = \frac{1}{\sqrt{8}}(2,0,0,0).
\]
In $K^*/K$, the element $e_1+K$ has order $2$, while $e_2+K$ and $e_3+K$ have order $4$, and these three classes generate a subgroup $G$ isomorphic to $\Z/2\Z \times \Z/4\Z \times \Z/4\Z$. We write elements of $G$ as $x_1e_1+x_2e_2+x_3e_3+K$ with $x_1 \in \Z/2\Z$ and $x_2,x_3 \in \Z/4\Z$. In these coordinates, the image of $A$ in $K^*/K$ becomes
\[
\{(0,0,0),\ (0,0,1),\ (0,1,0),\ (1,1,1)\}.
\]

In the finite abelian group setting of Definition~2.9 in \cite{CKRS}, Li, Pott, and Sch\"uler's Example~3.22 \cite{LiPottSchuler} states that the subset
\[
S = \{(0,0,0),\ (0,0,1),\ (0,1,0),\ (1,1,1)\}
\]
of $\Z/2\Z \times \Z/4\Z \times \Z/4\Z$ is formally dual to
\[
T = \{(0,0,0),\ (0,0,1),\ (0,1,0),\ (0,1,1),\ (1,0,0),\ (1,0,3),\ (1,3,0),\ (1,3,3)\}
\]
in the character group.

To make use of this duality, we map elements of $K^*/K$ to elements of the dual group $\widehat{G}$ of characters of $G$ by using the perfect discriminant pairing
\[
[x+K,y+K]=e^{2\pi i\langle x,y\rangle}
\]
on $K^*/K$. Let
\[
f_1=\frac{1}{\sqrt{8}}(0,0,2,2), \quad f_2=\frac{1}{\sqrt{8}}(0,0,2,0), \quad f_3=\frac{1}{\sqrt{8}}(1,-1,1,1),
\]
which are in $K^*$ and satisfy
\[
\langle f_i,e_j\rangle \in
\begin{cases}
1/2+\Z & \text{if $i=j=1$,}\\
1/4+\Z & \text{if $i=j\in \{2,3\}$, and}\\
\Z & \text{if $i\ne j$.}
\end{cases}
\]
Thus pairing with
\[
r_{a,b,c}=a f_1+b f_2+c f_3
\]
via $x + K \mapsto [x+K, r_{a,b,c}+K]$ yields the character
\[
x_1e_1+x_2e_2+x_3e_3+K
\mapsto
(-1)^{a x_1} i^{b x_2+c x_3}
\]
of $G$.

Next let
\[
v = \frac{1}{\sqrt{8}}(4,2,2,0),
\]
which is in $K^*$. Since $K=\sqrt{2} D_4$ has discriminant $64$, the finite group $K^*/K$ has order $64$, and the pairing on $K^*/K$ is perfect. The subgroup $G$ has order $32$, so its annihilator $G^\perp$ under the pairing has order $2$. The vector $v$ is not in $K$ but has integral inner product with $e_1,e_2,e_3$, and hence pairs trivially with $G$. Therefore $G^\perp=\{0,v+K\}$, and each character of $G$ has precisely two preimages under the restriction map $\rho\colon K^*/K\to \widehat{G}$ induced by the pairing, namely $r_{a,b,c}+K$ and $r_{a,b,c}+v+K$.

Let
\[
A^* = \{r_{a,b,c} : (a,b,c) \in T\} \cup \{r_{a,b,c}+v : (a,b,c) \in T\} \subseteq K^*,
\]
so that the image of $A^*$ in $K^*/K$ is precisely $\rho^{-1}(T)$. Applying Lemma~4.1 of \cite{CKRS} to the formally dual pair $S$ and $T$ shows that the images of $A$ and $A^*$ in $K^*/K$ are formally dual.

\subsection{The formal dual of Vardy's packing}

Now we can define the formal dual of $\cV_{20}$ by
\[
\cV^*_{20}
= \{\pi_V(w) : w \in \Lambda_{24} \text{ and } \pi_U(w) \in A^*\}.
\]
The gluing isomorphism $\gamma\colon K^*/K \to L^*/L$ sends the residue classes represented by $A$ and $A^*$ to the cosets of $L$ occurring in $\cV_{20}$ and $\cV^*_{20}$, respectively. As noted above, the map $\gamma$ is an anti-isometry of the discriminant pairing, but the resulting minus sign only complex conjugates the character sums, while their squared absolute values and hence formal duality are unchanged. Theorem~2.8 in \cite{CKRS} then shows that $\cV^*_{20}$ is formally dual to $\cV_{20}$.

The lattice $L$ has discriminant $64$ and hence covolume $8$. Thus $\cV_{20}$ and $\cV^*_{20}$, which are unions of $4$ and $16$ translates of $L$, have particle densities $1/2$ and $2$, respectively. Since $\cV_{20}$ has minimal squared distance $7/2$, formal duality for periodic configurations implies that the first nonzero Bragg peak of $\cV^*_{20}$ occurs at squared length $7/2$. After rescaling $\cV^*_{20}$ to unit particle density, its stealth radius is
\[
2^{-1/20}\sqrt{\frac{7}{2}}.
\]

Since $\cV_{20}$ is a denser sphere packing than any known lattice packing in $\R^{20}$, its formal dual is stealthier than any known lattice at unit particle density. Specifically, the best lattice packing known in $\R^{20}$ is the laminated lattice $\Lambda_{20}$, whose sphere-packing center density is $1/8$ (see Table~I.1 in \cite{SPLAG}). The unit-particle-density rescaling of the dual of a lattice packing in $\R^d$ with center density $\delta$ has stealth radius $2\,\delta^{1/d}$. Thus the rescaling of $\Lambda_{20}^*$ has stealth radius
\[
2\cdot (1/8)^{1/20}=2^{17/20},
\]
and indeed
\[
2^{-1/20} \sqrt{\frac{7}{2}} = 1.80710140\ldots > 1.80250092\ldots = 2^{17/20}.
\]

\section*{Use of artificial intelligence}

The details of the $20$-dimensional stealthy configuration were worked out by the GPT-5.5~Pro model in ChatGPT. We prompted it to search for a stealthy configuration using Conway and Sloane's antipode construction, and it found an example in $\R^{20}$. We subsequently asked it to simplify this configuration and check whether it was formally dual to Vardy's sphere packing. The authors then verified the calculations. In a later exchange, ChatGPT observed that this construction is closely related to Example~3.22 in \cite{LiPottSchuler}. We do not know to what extent this paper influenced the model's output in the original exchange. ChatGPT also provided editorial suggestions and feedback on the manuscript. Except for the AI assistance described above, the authors developed the arguments and wrote the manuscript, and we take responsibility for the correctness of its contents.

\section*{Acknowledgments}

A.~Cohen was partly supported by a Hertz Fellowship and an internship at Microsoft Research New England. H.~Cohn thanks OpenAI for providing a ChatGPT Pro subscription at no cost. M.~Viazovska's work was partly funded by Swiss National Science Foundation grant number 215337 on ``Sphere packing, energy minimization and Fourier interpolation.''

\end{document}